\documentclass[12pt]{amsart}

\usepackage{latexsym}
\usepackage{psfrag}
\usepackage{amsmath}
\usepackage{amssymb}
\usepackage{epsfig}
\usepackage{amsfonts}
\usepackage{amscd}
\usepackage{mathrsfs}
\usepackage{graphicx}
\usepackage{enumerate}

\newtheorem{theorem}{Theorem}[section]
\newtheorem{lemma}[theorem]{Lemma}

\begin{document}%%%%%%%%%%%%%%%%%%%%%%%%%%%%%%%%%%%%%%%%%%%%%%%%%%%%%%%%
%%%%%%%%%%%%%%%%%%%%%%%%%%%%%%%%%%%%%%%%%%%%%%%%%%%%%%%%%%%%%%%%%%%%%%%%

\title[Andrews Style Identities]{Andrews Style Partition Identities}

\author[Kur\c{s}ung\"{o}z]{Ka\u{g}an Kur\c{s}ung\"{o}z}
\address{Faculty of Engineering and Natural Sciences, Sabanc{\i} University, \.{I}stanbul, Turkey}
\email{kursungoz@sabanciuniv.edu}

%    For articles to be published after 1 January 2010, you may use
%    the following version:
\subjclass[2010]{Primary 05A15, 05A17, 11P84, Secondary 05A19}

\keywords{Integer Partition, Rogers-Ramanujan-Gordon Identities}

\date{\today}

\begin{abstract}
\noindent
We propose a method to construct a variety of partition identities at once.  
The main application is an all-moduli generalization of some of Andrews' results in \cite{AndrewsParity}.  
The novelty is that the method constructs solutions to functional equations 
which are satisfied by the generating functions.  
In contrast, the conventional approach is to show that a variant of well-known series 
satisfies the system of functional equations, 
thus reconciling two separate lines of computations.  
\end{abstract}

\maketitle

%%%%%%%%%%%%%%%%%%%%%%%%%%%%%%%%%%%%%%%%%%%%%%%%%%%%%%%%%%%
%%%%                                                   %%%%
%%%%                   Introduction                    %%%%
%%%%                                                   %%%%
%%%%%%%%%%%%%%%%%%%%%%%%%%%%%%%%%%%%%%%%%%%%%%%%%%%%%%%%%%%

\section{Introduction}
\label{secIntro}

Rogers-Ramanujan identities is a milestone in integer partition theory.  
They are independently discovered by three mathematicians 
(Rogers~\cite{Rogers}, Ramanujan, in one of his letters to Hardy before 1913~\cite{Hardy}, Schur~\cite{Schur}).  
For purposes of this note, we present one of them in Schur's terminology.  

\begin{theorem}
\label{thRR1}
  Let $A(n)$ denote the number of partitions of $n$ into parts 
  that are $\not\equiv 0, \pm 2 \pmod{5}$.  
  Let $B(n)$ denote the number of partitions of $n$
  into distinct and non-consecutive parts.  
  Then $A(n) = B(n)$.  
\end{theorem} 

The condition imposed by $A(n)$ is called a divisibility condition, 
and that by $B(n)$ a multiplicity condition.  
Following the $q$-series notation in \cite{GR}, 
one can write a generating function 
for partitions enumerated by $A(n)$ easily, and the theorem becomes

\begin{equation*}
%\label{eqRR1}
  \sum_{n \geq 0} B(n)q^n = \frac{(q^2, q^3, q^5; q^5)_\infty}{(q; q)_\infty}
  \textrm{.  }
\end{equation*}
Above and elsewhere, 
\[
  (a; q)_n = (1 - a)(1 - aq) \cdots (1 - aq^{n-1}), 
\]
\[
  (a_1, \ldots, a_r; q)_n = (a_1; q)_n \cdots (a_r; q)_n, 
\]
\[
  (a; q)_\infty = \lim_{n \to \infty} (a; q)_n.  
\]
The infinite products converge absolutely for $\vert q \vert < 1$.  

It was not until 1960 did Gordon give an extension to Rogers-Ramanujan identities
where the multiplicity conditions involved consecutive pairs only, 
but not consecutive triples, or parts that are two or more apart~\cite{RRG}.  
One in Gordon's family of identities is the following.  

\begin{theorem}
\label{thRRG1}
  Let $A_{k, a}(n)$ denote the number of partitions of $n$ into parts 
  that are $\not\equiv 0, \pm a \pmod{2k+1}$.  
  Let $B_{k, a}(n)$ denote the number of partitions of $n$
  where the combined number of occurrences of a consecutive pair of parts 
  is at most $k-1$, and $1$ occurs at most $a-1$ times.  
  Then $A_{k, a}(n) = B_{k, a}(n)$.  
\end{theorem} 
As above, we can express the theorem as 
\begin{equation*}
%\label{eqRRG1}
  \sum_{n \geq 0} B_{k, a}(n)q^n = \frac{(q^a, q^{2k+1-a}, q^{2k+1}; q^{2k+1})_\infty}{(q; q)_\infty}
  \textrm{.  }
\end{equation*}
Notice that $k = a = 2$ yields the first of the Rogers-Ramanujan identities.  

Recently, Andrews revisited the problem \cite{AndrewsParity}, 
and among the generalizations he gave was the following.  

\begin{theorem}
\label{thAndrewsParity1}
  For $a \equiv k \pmod{2}$, 
  let $W_{k, a}(n)$ denote the number of partitions of $n$ 
  subject to Gordon's multiplicity restriction (Theorem \ref{thRRG1}) 
  and even parts appear an even number of times.  
  Then
  \begin{equation*}
  %\label{eqAndrewsParity1}
    \sum_{n \geq 0} W_{k, a}(n) q^n 
    = \frac{(-q; q^2)_\infty (q^a, q^{2k+2-a}, q^{2k+2}; q^{2k+2})_\infty}{(q^2; q^2)_\infty}.  
  \end{equation*}
\end{theorem}

It is possible to interpret the infinite product on the right hand side 
in terms of partitions \cite{AndrewsParity}.  
Andrews topped the paper off with a list of open questions, 
some of which have been solved by various authors.  
For instance, a missing case in the group of results Theorem \ref{thAndrewsParity1} belong 
was supplied by Kim and Yee \cite{KimYee}. 
They also gave extensions of Andrews' results 
for higher moduli instead of parity.  

The aim of this paper is to develop a computational method to produce 
% and at the same time prove 
similar partition identities.  
The conventional way of proving many partition identities is 
finding a system of functional equations first.  
Those equations are satisfied by the generating function of some class of partitions.  
Then, a twist of a well-known series (almost always Andrews' $J$-function~\cite[Ch. 7]{AndrewsBlueBook})
is shown to satisfy the same set of equations, hence an identity is obtained.  
The proposed method is linear in the sense that once the functional equations are written, 
their solutions are constructed from scratch, 
so the second parts of proofs are not independent computations.  

The method is a further exploitation of the ideas Andrews presented in~\cite{AndrewsPosetRR}.  
As the main application, we will give another all-moduli generalization to Andrews' aforementioned identities.  
The method will also give a unified proof to some of Andrews' parity results~\cite{AndrewsParity}, and 
Kim and Yee's addendum~\cite{KimYee}.  
After some preliminaries in section \ref{secPrelim}, 
we will obtain infinitely many families of unusual identities in section \ref{secMain}.  
Although the proofs in section \ref{secMain} are complete, 
their actual mechanisms will be reverse-engineered in section \ref{secConstruction}, 
which is the point of this paper.  
We conclude with further topics of research and open problems in section \ref{secFutureResearch}.  

\section{Preliminaries}
\label{secPrelim}

A partition of a non-negative integer $n$ is a non-increasing sequence of positive integers
the sum of which is $n$.  
\[
  5 + 3 + 2 + 2 + 1
\]
is a partition of 13 into 5 parts.  
Another way to represent partitions is the frequency notation.  
For all $i>0$, $f_i$ is the number of occurrences of $i$ in the given partition.  
In the above example, 
\[
  f_1 = 1, f_2 = 2, f_3 = 1, f_4 = 0, f_5 = 1, \textrm{ and } f_i = 0 \textrm{ for } i > 5.  
\]
Frequencies with negative indices are zero by definition.  
With the frequency notation, it becomes very convenient to express the 
class of theorems in section \ref{secIntro} \cite[ch. 7]{AndrewsBlueBook}.  

\begin{theorem}[Gordon~\cite{RRG}]
\label{thmRRGfull}
  Let $a, k, i, n$ be non-negative integers such that $1 \leq a \leq k$.  
  If $B_{k, a}(n)$ be the number of partitions of $n$ where 
  \[
    f_i + f_{i+1} < k \textrm{, and } f_1 < a, 
  \]
  then 
  \begin{equation*}
  %\label{eqRRGfull}
    \sum_{n \geq 0} B_{k, a}(n)q^n = \frac{(q^a, q^{2k+1 - a}, q^{2k+1}; q^{2k+1})_\infty}{(q; q)_\infty}
    \textrm{.  }
  \end{equation*}
\end{theorem}

\begin{theorem}[Andrews~\cite{AndrewsParity}]
\label{thmAndrewsOddEvenFull}
  Suppose $a, k, i, n$ are non-negative integers such that $1 \leq a \leq k$.  
  For $a \equiv k \pmod{2}$, 
  if $W_{k, a}(n)$ denote the number of partitions of $n$ such that 
  \[
    f_{i} + f_{i+1} < k, \quad f_1 < a \textrm{, and } 2 \vert f_{2i}, 
  \]
  then 
  \begin{equation*}
  %\label{eqAndrewsParityEvenFull}
    \sum_{n \geq 0} W_{k, a}(n) q^n 
    = \frac{(-q; q^2)_\infty (q^a, q^{2k+2-a}, q^{2k+2}; q^{2k+2})_\infty}{(q^2; q^2)_\infty}.  
  \end{equation*}
  For $k$ odd and $a$ even, 
  if $\overline{W}_{k,a}(n)$ denotes the number of partitions of $n$ where 
  \[
    f_{i} + f_{i+1} < k, \quad f_1 < a \textrm{, and } 2 \vert f_{2i+1}, 
  \]
  then 
  \begin{equation*}
  %\label{eqAndrewsParityOddFull}
    \sum_{n \geq 0} \overline{W}_{k, a}(n) q^n 
    = \frac{(q^a, q^{2k+2-a}, q^{2k+2}; q^{2k+2})_\infty}{(-q; q^2)_\infty (q; q)_\infty}.  
  \end{equation*}
\end{theorem}

Notice that it is no longer straightforward to interpret the latter infinite product 
in terms of partitions.  
This property will pertain to many other identities to follow, 
so we will abstain from writing the divisibility conditions for 
partitions generated by the infinite products.  

We conclude this section with the missing case in Theorem \ref{thmAndrewsOddEvenFull}.  

\begin{theorem}[Kim and Yee \cite{KimYee}]
\label{thKimYee}
  Suppose $a, k, i, n$ are non-negative integers such that 
  $1 \leq a \leq k$, and $a \not \equiv k \pmod{2}$.  
  Let $W_{k,a}(n)$ denote the number of partitions of $n$ where 
  \[
    f_{i} + f_{i+1} < k, \quad f_1 < a \textrm{, and } 2 \vert f_{2i}. 
  \]
  Then, 
  \begin{eqnarray*}
  %\label{eqKimYee}
    \sum_{n \geq 0} W_{k, a}(n) q^n 
    = & \frac{1}{1+q} \frac{(-q; q^2)_\infty (q^{a+1}, q^{2k+1-a}, q^{2k+2}; q^{2k+2})_\infty}{(q^2; q^2)_\infty}
      %\nonumber 
      \\
    + & \frac{q}{1+q} \frac{(-q; q^2)_\infty (q^{a-1}, q^{2k+3-a}, q^{2k+2}; q^{2k+2})_\infty}{(q^2; q^2)_\infty}.  
  \end{eqnarray*}
\end{theorem}

Observe that it is still possible to interpret the identity using multiplicity and divisibility conditions.  

\section{Main Results}
\label{secMain}

This section is devoted to establishing the following identities.  

\begin{theorem}
\label{thmOddEvenFull}
  Suppose $a, k, d, e, f, i, n$ are non-negative integers such that 
  $1 \leq a, d \leq k$, $1 \leq f \leq d$, 
  and $e = d$ or $2e = d$.  
  Let ${}_dB_{dk+e, da + f}(n)$ be the number of partitions of $n$ such that 
  \[
    f_i + f_{i+1} < dk + e, \quad f_1 < da + f \textrm{, and } \quad d \vert f_{2i}, 
  \]
  and let ${}_d\overline{B}_{dk+e, da}(n)$ be the number of partitions of $n$ such that 
  \[
    f_i + f_{i+1} < dk + e, \quad f_1 < da \textrm{, and } \quad d \vert f_{2i + 1}.  
  \]
  Then, 
  \begin{equation*}
  %\label{eqEvenFull}
    \sum_{n \geq 0} {}_dB_{dk+e, da + f}(n) q^n = 
  \end{equation*}
  \begin{equation*}
  \begin{cases}
      \begin{array}{rl}
	\frac{1}{(q^{2d}; q^{2d})_\infty (q; q^2)_\infty} & 
	\left( \frac{(1 - q^{d - e + f})}{(1 - q^d)} (q^{da+e}, q^{2dk-da+e+d}, q^{2dk+2e+d}; q^{2dk+2e+d})_\infty \right.\\ & \\
	& \left. + \frac{(q^{d - e + f} - q^d)}{(1 - q^d)} (q^{d(a-1)+e}, q^{2dk-da+e+2d}, q^{2dk+2e+d}; q^{2dk+2e+d})_\infty \right)
      \end{array}, 
      & \textrm{ if } f < e \\
      \phantom{0} & \phantom{0} \\
      \frac{1}
	{(q^{2d}; q^{2d})_\infty (q; q^2)_\infty}
	(q^{da+e}, q^{2dk-da+e+d}, q^{2dk+2e+d}; q^{2dk+2e+d})_\infty, 
      & \textrm{ if } f = e \\
      \phantom{0} & \phantom{0} \\
      \begin{array}{rl}
	\frac{1}{(q^{2d}; q^{2d})_\infty (q; q^2)_\infty} & 
	\left( \frac{(q^{f-e} - q^d)}{(1 - q^d)} (q^{da+e}, q^{2dk-da+e+d}, q^{2dk+2e+d}; q^{2dk+2e+d})_\infty \right.\\ & \\
	& \left. + \frac{(1 - q^{f-e})}{(1 - q^d)} (q^{d(a+1)+e}, q^{2dk-da+e}, q^{2dk+2e+d}; q^{2dk+2e+d})_\infty \right)
      \end{array},
      & \textrm{ if } f > e, 
    \end{cases}
  \end{equation*}
  and 
  \begin{equation*}
  %\label{eqOddFull}
    \sum_{n \geq 0} {}_d\overline{B}_{dk+e, da}(n) q^n 
    =       \frac{1}
	{(q^{d}; q^{2d})_\infty (q^2; q^2)_\infty}
	(q^{da}, q^{2dk-da+2e+d}, q^{2dk+2e+d}; q^{2dk+2e+d})_\infty.  
  \end{equation*}
\end{theorem}

In other words, ${}_dB_{dk+e, da + f}(n)$ (respectively, ${}_d\overline{B}_{dk+e, da + f}(n)$)
enumerates the number of partitions of $n$ 
satisfying Gordon's criterion, in which each even (respectively, odd) part appears a multiple of $d$ times.  
Observe that $d = 1$ is Rogers-Ramanujan-Gordon identities (Theorem \ref{thRRG1}), 
and $d = 2$ is Andrews' theorems (Theorem \ref{thmAndrewsOddEvenFull}) 
and Kim and Yee's addendum (Theorem \ref{thKimYee}).  
Unless the parameters satisfy some additional conditions, 
it is not immediately possible to interpret the infinite products in terms of partitions.  

To prove the theorem, we need auxiliary series and an intermediate result.  
The parameters being as in Theorem \ref{thmOddEvenFull}, define
\begin{equation}
\label{eqEvenSeries}
  {}_dC_{dk+e, da+f}(x; q) 
  := \displaystyle \sum_{n \geq 0} {}_dc\alpha[f]_n(x; q) \; \left( q^{-n} \right)^{da+f} 
    + {}_dc\beta[f]_n(x; q) \; \left( xq^{n+1} \right)^{da+f}, 
\end{equation}
and 
\begin{equation}
\label{eqOddSeries}
  {}_dT_{dk+e, da}(x; q) 
  := \displaystyle \sum_{n \geq 0} {}_dt\alpha_n(x; q) \; \left( q^{-n} \right)^{da} 
    + {}_dt\beta_n(x; q) \; \left( xq^{n+1} \right)^{da}, 
\end{equation}
where the terms in the single sums are given as follows.  
\begin{equation*}
  {}_dc\alpha[f]_n(x; q) 
  = \frac{ (-1)^n x^{(dk+e)n} q^{(2dk+2e+d)\binom{n+1}{2} + n(f-e)} \left( (xq)^{d} ; q^{2d} \right)_\infty }
	{ (q^d; q^d)_n \left( (xq^{n+1})^{d}; q^d \right)_\infty (xq; q^2)_\infty }
  %\cdots
\end{equation*}
\begin{equation*}
  \times
  \begin{cases}
    \frac{1 - (xq)^{d+f-e} }{ 1 - (xq)^d } + q^{nd} \; \frac{(xq)^{d+f-e} - (xq)^{d} }{ 1 - (xq)^d }, 
    & \textrm{ if } f < e \\
    \phantom{0} & \phantom{0} \\
    1, & \textrm{ if } f = e \\
    \phantom{0} & \phantom{0} \\
    \frac{ (xq)^{f-e} - (xq)^{d} }{ 1 - (xq)^d } + q^{-nd} \; \frac{ 1 - (xq)^{f-e} }{ 1 - (xq)^d }, 
    & \textrm{ if } f > e, 
  \end{cases}
\end{equation*}

\begin{equation*}
  {}_dc\beta[f]_n(x; q) 
  = - \frac{ (-1)^n x^{(dk+e)n} q^{(2dk+2e+d)\binom{n+1}{2} + n(e-f)} \left( (xq)^{d} ; q^{2d} \right)_\infty }
	{ (q^d; q^d)_n \left( (xq^{n+1})^{d}; q^d \right)_\infty (xq; q^2)_\infty }
  %\cdots
\end{equation*}

\begin{equation*}
  \times
  \begin{cases}
    \frac{1 - (xq)^{e-f} }{ 1 - (xq)^d } + q^{-nd} \; \frac{ 1 - (xq)^{e-f} }{ 1 - (xq)^d }, 
    & \textrm{ if } f < e \\ 
    \phantom{0} & \phantom{0} \\
    1, & \textrm{ if } f = e \\ 
    \phantom{0} & \phantom{0} \\
    \frac{ 1 - (xq)^{d-f+e} }{ 1 - (xq)^d } + q^{nd} \; \frac{ (xq)^{d-f+e} - (xq)^{d} }{ 1 - (xq)^d }, 
    & \textrm{ if } f > e, 
  \end{cases}
\end{equation*}

\begin{equation*}
  {}_dt\alpha_n(x; q) 
  = \frac{ (-1)^n x^{(dk+e)n} q^{(2dk+2e+d)\binom{n+1}{2}} \left( (xq^2)^{d} ; q^{2d} \right)_\infty }
	{ (q^d; q^d)_n \left( (xq^{n+1})^{d}; q^d \right)_\infty (xq^2; q^2)_\infty }, 
\end{equation*}

\begin{equation*}
  {}_dt\beta_n(x; q) 
  = - \frac{ (-1)^n x^{(dk+e)n} q^{(2dk+2e+d)\binom{n+1}{2}} \left( (xq^2)^{d} ; q^{2d} \right)_\infty }
	{ (q^d; q^d)_n \left( (xq^{n+1})^{d}; q^d \right)_\infty (xq^2; q^2)_\infty }.  
\end{equation*}

It is true that the $\alpha$'s and $\beta$'s depend on $k$ and $e$ also, 
but they are somewhat less significant parameters than $d$, 
so we will not explicitly mention that bond in order to avoid a profusion of indices.  
However, it is important to keep in mind that neither kind of $\alpha$'s or $\beta$'s 
depend on $a$.  We will make use of this independence throughout the computations.  

\begin{lemma}
\label{lmOddEvenFull}
  Suppose $a, k, d, e, f, i, n, m$ are non-negative integers such that 
  $1 \leq a, d \leq k$, $1 \leq f \leq d$, 
  and $e = d$ or $2e = d$.  
  Let ${}_db_{dk+e, da + f}(m, n)$ be the number of partitions of $n$ 
  into $m$ parts such that 
  \[
    f_i + f_{i+1} < dk + e, \quad f_1 < da + f \textrm{, and } \quad d \vert f_{2i}, 
  \]
  and let ${}_d\overline{b}_{dk+e, da}(m, n)$ be the number of partitions of $n$ 
  into $m$ parts such that 
  \[
    f_i + f_{i+1} < dk + e, \quad f_1 < da \textrm{, and } \quad d \vert f_{2i + 1}.  
  \]
  Let
  \[
    {}_d\mathcal{B}_{dk+e, da+f}(x; q) 
    = \sum_{m,n \geq 0} {}_db_{dk+e, da + f}(m, n) x^m q^n ,
  \]
  and
  \[
    {}_d\overline{\mathcal{B}}_{dk+e, da}(x; q) 
    = \sum_{m,n \geq 0} {}_d\overline{b}_{dk+e, da}(m, n) x^m q^n .
  \]
  Then, 
  \begin{equation*}
  %\label{eqEvenGenFunc}
    \displaystyle {}_d\mathcal{B}_{dk+e, da+f}(x; q) = {}_dC_{dk+e, da+f}(x; q) 
  \end{equation*}
  and 
  \begin{equation*}
  %\label{eqOddGenFunc}
    \displaystyle {}_d\overline{\mathcal{B}}_{dk+e, da}(x; q) = {}_dT_{dk+e, da}(x; q).  
  \end{equation*}
\end{lemma}
\begin{proof}
  We show that both sides of the asserted identities 
  satisfy the same functional equations with the same initial conditions.  

  Observe that 
  \begin{align}
    {}_db_{dk+e, da+f}(m,n) & - {}_db_{dk+e, da+f-1}(m,n) \notag \\
  \label{eqEvenToOddRecurrence}
    = & \begin{cases}
       {}_d\overline{b}_{dk+e, dk-da+d}(m-(da+f-1),n-m) & \textrm{ if } f \leq e \\
       {}_d\overline{b}_{dk+e, dk-da}(m-(da+f-1),n-m) & \textrm{ if } f > e,
      \end{cases}
  \end{align}
  \begin{equation}
  \label{eqOddToEvenRecurrence}
    {}_d\overline{b}_{dk+e, da+d}(m,n) - {}_d\overline{b}_{dk+e, da}(m,n)
    = {}_db_{dk+e,dk-da+e}(m-da,n-m)
  \end{equation}
  for $n,m \geq 0$, 
  \begin{equation}
  \label{eqOddEvenInitCondNegatives}
    {}_db_{dk+e, da+f}(m,n) = {}_d\overline{b}_{dk+e, da}(m,n) = 0
  \end{equation}
  for $m < 0$, $n < 0$, or $m = 0$ and $n > 0$, or $m > 0$ and $n = 0$,
  \begin{equation}
  \label{eqOddEvenInitCondEmptyPtn}
    {}_db_{dk+e, da+f}(0,0) = {}_d\overline{b}_{dk+e, da}(0,0) = 1, 
  \end{equation}
  and finally
  \begin{equation}
  \label{eqOddEvenInitCondZero}
    {}_db_{dk+e, 0}(m,n) = {}_d\overline{b}_{dk+e, 0}(m,n) = 0 
  \end{equation}
  for all $m,n$.  
  
  To justify \eqref{eqEvenToOddRecurrence}, 
  let $\lambda = 1 f_1 + 2 f_2 + \cdots + r f_r$ 
  be a partition enumerated by 
  ${}_db_{dk+e, da+f}(m,n) - {}_db_{dk+e, da+f-1}(m,n)$.  
  Then, 
  \[
    f_i + f_{i+1} < dk+e, \quad d|f_{2i} \textrm{, and } f_1 = da + f - 1.  
  \]
  Erase the 1's from $\lambda$, and subtract 1 from the remaining parts.  
  This switches the parities of parts.  
  Call the remaining partition $\widetilde{\lambda}$, which has exactly $m - (da+f-1)$ parts 
  thanks to the deleted 1's.  
  Its weight, or sum of its parts is $n-m$, since we effectively subtracted 1 from all parts, 
  including the 1's.  
  On the one hand, in $\lambda$, $f_1 + f_2 < dk+e$ and $f_1 = da+f-1$, 
  so $f_2 < dk - da + e - f + 1$.  
  On the other hand, $d | f_2$, so $f_2 < dk - da + d$ or $f_2 < dk - da$
  depending on how $e$ and $f$ compare.  
  This last adjustment is to make both sides of the inequality multiples of $d$.  
  Now, frequencies of parts in $\widetilde{\lambda}$ satisfy 
  \[
    f_i + f_{i+1} < dk+e, \quad f_1 < dk - da + d \textrm{ or } f_1 < dk - da, 
    \textrm{ and } d \vert f_{2i+1}.   
  \]
  Therefore, $\widetilde{\lambda}$ is enumerated by 
  ${}_d\overline{b}_{dk+e, dk-da+d}(m-(da+f-1),n-m)$ or %\linebreak
  ${}_d\overline{b}_{dk+e, dk-da}(m-(da+f-1),n-m)$, 
  depending on comparison of $e$ and $f$.  
  Proof of \eqref{eqOddToEvenRecurrence} is similar.  
  
  We only partition non-negative integers into positive integers.  
  A partition of a positive integer must contain at least one part,   
  which explains \eqref{eqOddEvenInitCondNegatives}.  
  There is a unique partition of zero, namely the empty partition.  
  This is captured by \eqref{eqOddEvenInitCondEmptyPtn}.  
  Frequencies cannot be negative.  
  In particular, there are no partitions with $f_1 < 0$,
  hence \eqref{eqOddEvenInitCondZero}.  
  
  The recurrences and initial conditions \eqref{eqEvenToOddRecurrence}-\eqref{eqOddEvenInitCondZero} 
  uniquely determine ${}_db_{dk+e, da + f}(m, n)$ and \linebreak[4]
  ${}_d\overline{b}_{dk+e, da}(m,n)$'s.  
  On the generating function side, they yield the following.  
  \begin{equation*}
  %\label{eqEvenToOddGFrec}
    {}_d\mathcal{B}_{dk+e, da+f}(x; q) - {}_d\mathcal{B}_{dk+e, da+f-1}(x; q)
    = \begin{cases}
     (xq)^{da+f-1} {}_d\overline{\mathcal{B}}_{dk+e, dk-da+d}(xq; q) & \textrm{ if } f \leq e \\
     (xq)^{da+f-1} {}_d\overline{\mathcal{B}}_{dk+e, dk-da}(xq; q) & \textrm{ if } f > e
    \end{cases}
  \end{equation*}
  \begin{equation*}
  %\label{eqOddToEvenGFrec}
    {}_d\overline{\mathcal{B}}_{dk+e, da + d}(x; q) - {}_d\overline{\mathcal{B}}_{dk+e, da}(x; q)
    = (xq)^{da} {}_d\mathcal{B}_{dk+e, dk-da+e}(x; q)
  \end{equation*}
  \begin{equation*}
  %\label{eqOddEvenGFInitEmpty}
    {}_d\mathcal{B}_{dk+e, da+f}(0; q) = {}_d\overline{\mathcal{B}}_{dk+e, da}(0; q) = 1
  \end{equation*}
  \begin{equation*}
  %\label{eqOddEvenGFInitZero}
    {}_d\mathcal{B}_{dk+e, 0}(x; q) = {}_d\overline{\mathcal{B}}_{dk+e, 0}(x; q) = 0
  \end{equation*}
  Notice that these uniquely determine the generating functions.  
  
  To complete the proof, we verify the following.  
  \begin{align}
    {}_dC_{dk+e, da+f}(x; q) & - {}_dC_{dk+e, da+f-1}(x; q) \notag \\
  \label{eqEvenToOddSeriesrec}
  = & \begin{cases}
     (xq)^{da+f-1} {}_dT_{dk+e, dk-da+d}(xq; q) & \textrm{ if } f \leq e \\
     (xq)^{da+f-1} {}_dT_{dk+e, dk-da}(xq; q) & \textrm{ if } f > e
    \end{cases}
  \end{align}
  \begin{equation}
  \label{eqOddToEvenSeriesrec}
    {}_dT_{dk+e, da + d}(x; q) - {}_dT_{dk+e, da}(x; q)
    = (xq)^{da} {}_dC_{dk+e, dk-da+e}(x; q)
  \end{equation}
  \begin{equation}
  \label{eqOddEvenSeriesInitEmpty}
    {}_dC_{dk+e, da+f}(0; q) = {}_dT_{dk+e, da}(0; q) = 1
  \end{equation}
  \begin{equation}
  \label{eqOddEvenSeriesInitZero}
    {}_dC_{dk+e, 0}(x; q) = {}_dT_{dk+e, 0}(x; q) = 0
  \end{equation}
  To see \eqref{eqEvenToOddSeriesrec} and \eqref{eqOddToEvenSeriesrec}, 
  we check that 
  \[
    {}_dc\alpha[f]_n(x;q) (q^{-n})^{da+f} - {}_dc\alpha[f-1]_n(x;q) (q^{-n})^{da+f-1}
  \]
  \begin{equation}
  \label{eqDeltaAlphaEvenToOdd}
    = \begin{cases}
       (xq)^{da+f-1} {}_dt\beta_{n-1}(xq) (xq^{n+1})^{dk-da+d} & \textrm{ if } f \leq e \\
       (xq)^{da+f-1} {}_dt\beta_{n-1}(xq) (xq^{n+1})^{dk-da} & \textrm{ if } f > e, 
      \end{cases} 
  \end{equation}
  \[
    {}_dc\beta[f]_n(x;q) (xq^{n+1})^{da+f} - {}_dc\beta[f-1]_n(x;q) (xq^{n+1})^{da+f-1}
  \]
  \begin{equation}
  \label{eqDeltaBetaEvenToOdd}
    = \begin{cases}
       (xq)^{da+f-1} {}_dt\alpha_{n}(xq) (q^{-n})^{dk-da+d} & \textrm{ if } f \leq e \\
       (xq)^{da+f-1} {}_dt\alpha_{n}(xq) (q^{-n})^{dk-da} & \textrm{ if } f > e, 
      \end{cases} 
  \end{equation}
  \begin{equation}
  \label{eqDeltaAlphaOddToEven}
    {}_dt\alpha_n(x; q) (q^{-n})^{da+d} - {}_dt\alpha_n(x; q) (q^{-n})^{da} 
    = (xq)^{da} {}_dc\beta[e]_{n-1}(xq;q) (xq^{n+1})^{dk-da+e}, 
  \end{equation}
  \begin{equation}
  \label{eqDeltaBetaOddToEven}
    {}_dt\beta_n(x; q) (xq^{n+1})^{da+d} - {}_dt\beta_n(x; q) (xq^{n+1})^{da} 
    = (xq)^{da} {}_dc\alpha[e]_{n}(xq;q) (q^{-n})^{dk-da+e}, 
  \end{equation}
  These are straightforward, case by case computations.  
  
  \eqref{eqOddEvenSeriesInitEmpty} follows from the fact that 
  \begin{equation*}
  %\label{eqAlphaZero}
    {}_dc\alpha[f]_n(0; q) (q^{-n})^{da+f}
    = {}_dt\alpha_n(0; q) (q^{-n})^{da}
    = \begin{cases}
	1 & \textrm{ if } n = 0 \\
	0 & \textrm{ if } n > 0, 
      \end{cases}
  \end{equation*}
  and 
  \begin{equation*}
  %\label{eqBetaZero}
    {}_dc\beta[f]_n(0; q) (xq^{n+1})^{da+f}
    = {}_dt\beta_n(0; q) (xq^{n+1})^{da}
    = 0.  
  \end{equation*}
  Finally, \eqref{eqOddEvenSeriesInitZero} is a consequence of 
  \begin{equation*}
  %\label{eqEvenAlphaEqualsBeta}
    {}_dc\alpha[0]_n(x;q) = {}_dc\beta[0]_n(x; q), 
  \end{equation*}
  and 
  \begin{equation*}
  %\label{eqOddAlphaEqualsBeta}
    {}_dt\alpha_n(x;q) = {}_dt\beta_n(x; q).   
  \end{equation*}
  Unless $e = d$ or $2e = d$, the penultimate equation fails.  
  This is the only place in the proof where we need that restriction.  
\end{proof}

\begin{proof}[proof of Theorem \ref{thmOddEvenFull}]
  We prove the case $f<e$ in the former identity.  
  The other cases are completely analogous.  
  First, 
  \[
    {}_dB_{dk+e, da+f}(n) = \sum_{m \geq 0} {}_db_{dk+e, da+f}(m,n)
  \]
  since both sides enumerate the same kind of partitions.  
  On the right hand side they are classified according to the number of parts.  
  by Lemma \ref{lmOddEvenFull}, 
  \[
    \sum_{n \geq 0} {}_dB_{dk+e, da+f}(n) q^n 
    = \sum_{m, n \geq 0} {}_db_{dk+e, da+f}(m,n) q^n
    = C_{dk+e, da+f}(1; q) 
  \]
  \begin{align*}
    = & \frac{ (q^d; q^{2d})_\infty }{ (q^d; q^d)_\infty (q, q^2)_\infty } \frac{1}{(1 - q^d)} \\
    & \times \left\{ 
      \sum_{n \geq 0} (-1)^n q^{ (2dk+2e+d) \binom{n+1}{2} } q^{-n(e-f)} q^{-n(da+f)}
      \left[ (1 - q^{d+f-e} ) + q^{nd} (q^{d+f-e} - q^d) \right]
    \right. \\
      & - \left.
      \sum_{n \geq 0} (-1)^n q^{ (2dk+2e+d) \binom{n+1}{2} } q^{(n+1)(e-f)} q^{(n+1)(da+f)}
      \left[ (1 - q^{d+f-e} ) + q^{-(n+1)d} (q^{d+f-e} - q^d) \right]
    \right\}.  
  \end{align*}
  Negate the index in the first sum, and shift it in the second.  
  Note that $\binom{n+1}{2} = \binom{-n}{2}$.  
  \begin{align*}
    = & \frac{ (q^d; q^{2d})_\infty }{ (q^d; q^d)_\infty (q, q^2)_\infty } \frac{1}{(1 - q^d)} \\
    & \times \left\{ 
      \sum_{n \leq 0} (-1)^n q^{ (2dk+2e+d) \binom{n}{2} } q^{n(e-f)} q^{n(da+f)}
      \left[ (1 - q^{d+f-e} ) + q^{-nd} (q^{d+f-e} - q^d) \right]
    \right. \\
      & + \left.
      \sum_{n > 0} (-1)^n q^{ (2dk+2e+d) \binom{n}{2} } q^{n(e-f)} q^{n(da+f)}
      \left[ (1 - q^{d+f-e} ) + q^{-nd} (q^{d+f-e} - q^d) \right]
    \right\}
  \end{align*}
  \begin{align*}
    = & \frac{ (q^d; q^{2d})_\infty }{ (q^d; q^d)_\infty (q, q^2)_\infty } \frac{ (1 - q^{d+f-e} ) }{(1 - q^d)}
      \sum_{n = -\infty}^{\infty} (-1)^n q^{ (2dk+2e+d) \binom{n}{2} } (q^{da+e})^n \\
    + & \frac{ (q^d; q^{2d})_\infty }{ (q^d; q^d)_\infty (q, q^2)_\infty } \frac{ (q^{d+f-e} - q^d) }{(1 - q^d)}
      \sum_{n = -\infty}^{\infty} (-1)^n q^{ (2dk+2e+d) \binom{n}{2} } (q^{d(a-1)+e})^n
  \end{align*}
  Finally, use Jacobi's triple product identity~\cite[p.15]{GR} to finish the proof.  
  \begin{align*}
    = & \frac{ (q^d; q^{2d})_\infty }{ (q^d; q^d)_\infty (q, q^2)_\infty } \frac{ (1 - q^{d+f-e} ) }{(1 - q^d)}
      \left( q^{da+e}, q^{2dk - da + d + e }, q^{2dk+2e+d}; q^{2dk+2e+d} \right)_\infty \\
    + & \frac{ (q^d; q^{2d})_\infty }{ (q^d; q^d)_\infty (q, q^2)_\infty } \frac{ (q^{d+f-e} - q^d) }{(1 - q^d)}
      \left( q^{d(a-1)+e}, q^{2dk - da + 2d + e }, q^{2dk+2e+d}; q^{2dk+2e+d} \right)_\infty
  \end{align*}
\end{proof}

\section{Construction of the Series}
\label{secConstruction}

In \cite{AndrewsPosetRR}, Andrews proves that 
for $1 \leq a \leq k$,
\begin{equation*}
%\label{eqRRG}
  \sum_{m,n \geq 0} {}_1b_{a,k}(m,n) x^m  q^n = Q_{k,a}(x;q) 
  := \sum_{n \geq 0} \alpha_n(x;q) (q^{-n})^a + \beta_n(x;q) (xq^{n+1})^a, 
\end{equation*}
where 
\[
  \alpha_n(x; q) = - \beta(x;q) 
  = \frac{ (-1)^n x^{kn} q^{(2k+1)\binom{n+1}{2} } }{ (q; q)_n (xq^{n+1}; q)_\infty }. 
\]
by showing that both sides satisfy the same functional equations 
with the same initial conditions.  
This proof of Rogers-Ramanujan-Gordon identities was first given in~\cite{AndrewsRRGanalytic}, 
and the series $Q_{k,a}(x;q)$ first appeared in~\cite{RR-Qappearance}.  
But the proofs here resemble those in \cite{AndrewsPosetRR} than the others.  
The functional equations and initial conditions are 
\eqref{eqEvenToOddSeriesrec}-\eqref{eqOddEvenSeriesInitZero}
for $d = 1$, in which case ${}_1C_{k+1, a+1}(x; q) = {}_1T_{k+1, a+1}(x; q)$.  
The key computations in \cite{AndrewsPosetRR} are 
\begin{equation}
\label{eqDeltaAlphaOriginal}
  \alpha_n(x;q) (q^{-n})^a - \alpha_n(x;q) (q^{-n})^{a-1}
  = (xq)^{a-1} \beta_{n-1}(xq; q) (xq^{n+1})^{k-a+1}, 
\end{equation}
and 
\begin{equation}
\label{eqDeltaBetaOriginal}
  \beta_n(x;q) (xq^{n+1})^a - \beta_n(x;q) (xq^{n+1})^{a-1}
  = (xq)^{a-1} \alpha_{n}(xq; q) (q^{-n})^{k-a+1}, 
\end{equation}
yielding
\begin{equation*}
\label{eqQRecurrence}
  Q_{k,a}(x; q) - Q_{k,a-1}(x;q) = (xq)^{a-1} Q_{k, k-a+1}(xq; q).  
\end{equation*}
Obviously, $Q_{k,a}(0, q) = 1$.  
The fact that $\alpha_n(x; q) = -\beta_n(x;q)$ guarantees
$Q_{k, 0}(x;q) = 0$, and the proof is complete.  

Inspired by this, we make some initial guesses and construct the series from scratch.  
To find extensions of Rogers-Ramanujan-Gordon identities, 
the conventional approach is to show that some variant of $Q_{k,a}(x; q)$ solves the 
functional equations that derive from the combinatorial descriptions, 
thus reconciling two series to get an identity.  
The method here is linear in the sense that once we have the functional equations, 
the series are constructed.  
If one so wishes, the connection to $Q_{k,a}(x; q)$ can be made.  

Given definitions of ${}_db_{dk+e, da+f}(m, n)$ and ${}_d\overline{b}_{dk+e, da+f}(m, n)$, 
the equations \eqref{eqEvenToOddSeriesrec}-\eqref{eqOddEvenSeriesInitZero} follow.  
We guess the generating functions to be of the form 
\begin{equation*}
%\label{eqEvenSeriesCrude}
  {}_dC_{dk+e, da+f}(x; q) 
  := \displaystyle \sum_{n \geq 0} {}_dc\alpha_n(x; q) \; \left( q^{-n} \right)^{da+f} 
    + {}_dc\beta_n(x; q) \; \left( xq^{n+1} \right)^{da+f}, 
\end{equation*}
\begin{equation*}
%\label{eqOddSeriesCrude}
  {}_dT_{dk+e, da}(x; q) 
  := \displaystyle \sum_{n \geq 0} {}_dt\alpha_n(x; q) \; \left( q^{-n} \right)^{da} 
    + {}_dt\beta_n(x; q) \; \left( xq^{n+1} \right)^{da}.  
\end{equation*}
One heuristic reason for having two types of terms in the series is that 
in recurrences \eqref{eqDeltaAlphaOriginal} and \eqref{eqDeltaBetaOriginal}, 
the $a$'s in the exponents on both sides simplify, 
and a double recurrence that define $\alpha_n(x; q)$ and $\beta_n(x; q)$ 
as nice infinite products is obtained.  

However, a few trials to find similar recurrences among 
${}_dc\alpha_n(x; q)$, ${}_dt\alpha_n(x; q)$,
${}_dc\beta_n(x; q)$, and ${}_dt\beta_n(x; q)$'s 
will lead to inconsistent equations, 
indicating that more freedom is needed.  
The remedy is the fact that there are exactly $d$ functional equations 
in \eqref{eqEvenToOddSeriesrec} that have ${}_dT_{dk+e, da}(xq; q)$ on the right hand side.  
Therefore, we use separate ${}_dc\alpha[f]_n(x; q)$ and \linebreak[4]
${}_dc\beta[f]_n(x; q)$'s, 
depending on the residue class $f\pmod{d}$, hence the forms of series as in 
\eqref{eqEvenSeries} and \eqref{eqOddSeries}.  

In order to achieve \eqref{eqEvenToOddSeriesrec} and \eqref{eqOddToEvenSeriesrec}, 
we require \eqref{eqDeltaAlphaEvenToOdd} - \eqref{eqDeltaBetaOddToEven}.  
We keep in mind that $f$ is treated as a residue class $\pmod{d}$, so 
${}_dc\alpha[0]_n(x; q) = {}_dc\alpha[d]_n(x; q)$.  
After simplifications, we have 
\begin{equation*}
%\label{eqDeltaAlphaEvenToOdd-Ia}
  q^{-n} {}_dc\alpha[f]_n(x; q) - {}_dc\alpha[f-1]_n(x; q)
  = \begin{cases}
      (xq^{n+1})^{dk+d+f-1} {}_dt\beta_{n-1}(xq; q) & \textrm{ if } f \leq e \\
      (xq^{n+1})^{dk+f-1} {}_dt\beta_{n-1}(xq; q) & \textrm{ if } f > e, 
    \end{cases} 
\end{equation*}
\begin{equation*}
%\label{eqDeltaBetaEvenToOdd-Ib}
  xq^{n+1} {}_dc\beta[f]_n(x; q) - {}_dc\beta[f-1]_n(x; q)
  = \begin{cases}
      (q^{-n})^{dk+d+f-1} {}_dt\alpha_{n}(xq; q) & \textrm{ if } f \leq e \\
      (q^{-n})^{dk+f-1} {}_dt\alpha_{n}(xq; q) & \textrm{ if } f > e, 
    \end{cases} 
\end{equation*}
\begin{equation*}
%\label{eqDeltaAlphaOddToEven-IIa}
  {}_dt\alpha_n(x; q)
  = \frac{ (xq^{n+1})^{dk+e} }{ (q^{-dn} - 1) } {}_dc\beta[e]_{n-1}(xq; q), 
\end{equation*}
and
\begin{equation*}
%\label{eqDeltaBetaOddToEven-IIb}
  {}_dt\beta_n(x; q)
  = \frac{ (q^{-n})^{dk+e} }{ ( (xq^{n+1})^d - 1) } {}_dc\alpha[e]_{n}(xq; q).  
\end{equation*}
The former pair of systems of equations 
can be rewritten using matrices.  
\begin{align*}
%\label{eqDeltaAlphaEvenToOdd-matrixIa}
  \begin{bmatrix}
    q^{-n} & 0 & \cdots & 0 & 1 \\
    -1 & q^{-n} & \cdots & 0 & 0 \\
    \vdots & & \ddots & & \\
    0 & 0 & & q^{-n} & 0 \\
    0 & 0 & & -1 & q^{-n} 
  \end{bmatrix}
  & \begin{bmatrix}
    {}_dc\alpha[0]_n(x; q) \\ \vdots \\ {}_dc\alpha[d-1]_n(x; q)
  \end{bmatrix} \\
  & = (xq^{n+1})^{dk} {}_dt\beta_{n-1}(xq; q) 
  \begin{bmatrix}
   1 \cdot (xq^{n+1})^d \\ \vdots \\ (xq^{n+1})^{e-1} \cdot (xq^{n+1})^d \\ 
   (xq^{n+1})^{e} \cdot 1 \\ \vdots \\ (xq^{n+1})^{d-1} \cdot 1
   \end{bmatrix}
\end{align*}
\begin{align*}
\label{eqDeltaBetaEvenToOdd-matrixIb}
  \begin{bmatrix}
    xq^{n+1} & 0 & \cdots & 0 & 1 \\
    -1 & xq^{n+1} & \cdots & 0 & 0 \\
    \vdots & & \ddots & & \\
    0 & 0 & & xq^{n+1} & 0 \\
    0 & 0 & & -1 & xq^{n+1} 
  \end{bmatrix}
  & \begin{bmatrix}
    {}_dc\beta[0]_n(x; q) \\ \vdots \\ {}_dc\beta[d-1]_n(x; q)
  \end{bmatrix} \\
  & = (q^{-n})^{dk} {}_dt\alpha_{n}(xq; q) 
  \begin{bmatrix}
   1 \cdot (q^{-n})^d \\ \vdots \\ (q^{-n})^{e-1} \cdot (q^{-n})^d \\ 
   (q^{-n})^{e} \cdot 1 \\ \vdots \\ (q^{-n})^{d-1} \cdot 1
   \end{bmatrix} 
\end{align*}
The displayed matrices have inverses
\begin{equation*}
\label{eqMatrixInversesAlpha}
  \frac{1}{ (1 - q^{dn}) }
  \begin{bmatrix}
   q^n & q^{dn} & q^{(d-1)n} & \cdots & & \\
   q^{2n} & q^n & q^{dn} & \cdots & & \\
   \vdots & & & \ddots & & \\
   & & & & q^{2n} & q^n
  \end{bmatrix}, 
\end{equation*}
and 
\begin{equation*}
\label{eqMatrixInversesBeta}
  \frac{ 1 }{ ((xq^{n+1})^d - 1) }
  \begin{bmatrix}
    (xq^{n+1})^{d-1} & 1 & (xq^{n+1}) & \cdots & \\
    (xq^{n+1})^{d-2} & (xq^{n+1})^{d-1} & 1 & \cdots & \\
    \vdots & & \ddots & \\
    1 & (xq^{n+1}) & \cdots & & (xq^{n+1})^{d-1}
  \end{bmatrix}, 
\end{equation*}
respectively.  Multiplying both sides of the matrix equations by the respective inverse matrix, 
we obtain the following recurrences.  
\begin{align*}
%\label{eqDeltaEvenAlphaToOddBeta}
  {}_dc\alpha[f]_n(x; q)  = & \frac{ (xq^{n+1})^{dk} {}_dt\beta_{n-1}(xq; q) }{ (1 - q^{nd}) } \notag \\ & \\ 
   & \times \begin{cases}
      \frac{ (xq)^e - (xq)^{d+f} }{ (1 - xq) } q^{ n(d+f) } 
      + \frac{ (xq)^{d+f} - (xq)^{d+e} }{ (1 - xq) } q^{ n(2d+f) } 
      & \textrm{ if } f < e \\ & \\ 
      \frac{ (xq)^e - (xq)^{d+e} }{ (1 - xq) } q^{ n(d+e) }
      & \textrm{ if } f = e \\ & \\ 
      \frac{ (xq)^f - (xq)^{d+e} }{ (1 - xq) } q^{ n(d+f) } 
      + \frac{ (xq)^{e} - (xq)^{f} }{ (1 - xq) } q^{ nf } 
      & \textrm{ if } f > e \\
    \end{cases}
\end{align*}

\begin{align*}
%\label{eqDeltaEvenBetaToOddAlpha}
  {}_dc\beta[f]_n(x; q) = & \frac{ (q^{-n})^{dk} {}_dt\alpha_{n}(xq; q) }{ ( (xq^{n+1})^{d} - 1 ) } \\ & \\ 
  & \times \begin{cases}
      \frac{ (xq)^{e-f} - (xq)^{d} }{ (1 - xq) } q^{ -nf } 
      + \frac{ 1 - (xq)^{e-f} }{ (1 - xq) } q^{ -n(d+f) } 
      & \textrm{ if } f < e \\ & \\ 
      \frac{ 1 - (xq)^{d} }{ (1 - xq) } q^{ -ne }
      & \textrm{ if } f = e \\ & \\ 
      \frac{ 1 - (xq)^{d-f+e} }{ (1 - xq) } q^{ -nf } 
      + \frac{ (xq)^{d-f+e} - (xq)^{d} }{ (1 - xq) } q^{ -n(f-d) } 
      & \textrm{ if } f > e \\
    \end{cases}
\end{align*}

By means of iteration, we first obtain 
\begin{equation*}
%\label{eqOddAlphaCrude}
  {}_dt\alpha_n(x; q) = \frac{ (-1)^n x^{(dk+e)n} q^{(n^2+n)(dk+e)} q^{d\binom{n+1}{2}} ( (xq^2)^d; q^{2d})_n }
		    { (q^d; q^d)_n ((xq^{n+1})^d; q^d)_n ( xq^2; q^2)_n }
		 \quad {}_dt\alpha_0(xq^{2n}; q), 
\end{equation*}
\begin{equation*}
%\label{eqOddBetaCrude}
  {}_dt\beta_n(x; q) = \frac{ (-1)^n x^{(dk+e)n} q^{(n^2+n)(dk+e)} q^{d\binom{n+1}{2}} ( (xq^2)^d; q^{2d})_n }
		    { (q^d; q^d)_n ((xq^{n+1})^d; q^d)_n ( xq^2; q^2)_n }
		 \quad {}_dt\beta_0(xq^{2n}; q), 
\end{equation*}
\begin{align*}
%\label{eqEvenAlphaCrude}
  {}_dc\alpha[f]_n(x; q) = & - \, 
    \frac{ (-1)^n x^{(dk+e)n} q^{(n^2+n)(dk+e)} q^{d \binom{n+1}{2}} ((xq)^{d}; q^{2d})_n }
      { (q^d; q^d)_n ((xq^{n+1})^{d}; q^{d})_{n-1} (xq; q^2)_n }
     \; {}_dt\beta_0(xq^{2n-1}; q) \\ & \\
   & \times
   \begin{cases}
    \frac{ 1 - (xq)^{d+f-e} }{ 1 - (xq)^{d} } q^{n(f-e)}
    + \frac{ (xq)^{d+f-e} - (xq)^{d} }{ 1 - (xq)^{d} } q^{n(d+f-e)} 
    & \textrm{ if } f < e \\
    & \\
    1 & \textrm{ if } f = e \\ 
    & \\
    \frac{ (xq)^{f-e} - (xq)^{d} }{ 1 - (xq)^{d} } q^{n(f-e)}
    + \frac{ 1 - (xq)^{f-e} }{ 1 - (xq)^{d} } q^{n(f-d-e)} 
    & \textrm{ if } f > e, 
   \end{cases} 
\end{align*}
and 
\begin{align*}
%\label{eqEvenBetaCrude}
  {}_dc\beta[f]_n(x; q) = & - \, 
    \frac{ (-1)^n x^{(dk+e)n} q^{(n^2+n)(dk+e)} q^{d \binom{n+1}{2}} ((xq)^{d}; q^{2d})_{n+1} }
      { (q^d; q^d)_n ((xq^{n+1})^{d}; q^{d})_{n+1} (xq; q^2)_{n+1} }
     \; {}_dt\alpha_0(xq^{2n+1}; q) \\ & \\
   & \times
   \begin{cases}
    \frac{ (xq)^{e-f} - (xq)^{d} }{ 1 - (xq)^{d} } q^{n(e-f)}
    + \frac{ 1 - (xq)^{e-f} }{ 1 - (xq)^{d} } q^{n(e-d-f)} 
    & \textrm{ if } f < e \\
    & \\
    1 & \textrm{ if } f = e \\ 
    & \\
    \frac{ 1 - (xq)^{d-f+e} }{ 1 - (xq)^{d} } q^{n(e-f)}
    + \frac{ (xq)^{d-f+e} - (xq)^{d} }{ 1 - (xq)^{d} } q^{n(e-f+d)} 
    & \textrm{ if } f > e.  
   \end{cases} 
\end{align*}

The next set of constraints to meet is \eqref{eqOddEvenSeriesInitZero}.  
For this, we introduce part of our next assumption.  
\begin{align*}
%\label{eqAlphasEqualBetas}
  {}_dt\alpha_n(x; q) & = - {}_dt\beta_n(x; q) \\
  {}_dc\alpha[0]_n(x; q) & = -{}_dc\beta[0]_n(x; q) 
\end{align*}
Observe that this is a sufficient but not necessary condition for \eqref{eqOddEvenSeriesInitZero}.  
It brings
\begin{equation*}
%\label{eqAlphaBetaZeroRec1}
  {}_dt\alpha_0(xq^{2n}; q) = - {}_dt\beta_0(xq^{2n}; q), 
\end{equation*}

\begin{align*}
%\label{eqAlphaBetaZeroRec2Crude}
  & {}_dt\beta_0(xq^{2n-1}; q) \left\{ \frac{ 1 - (xq)^{d-e} }{ 1 - (xq)^d } q^{-ne} 
						+ \frac{ (xq)^{d-e} - (xq)^{d} }{ 1 - (xq)^d } q^{n(d-e)}\right\} \notag \\
  & \\
  & = - {}_dt\alpha_0(xq^{2n+1}; q) 
    \frac{ (1 - x^d q^{(2n+1)d}) }{ (1 - x^{d} q^{2nd}) (1 - x^{d} q^{(2n+1)d}) (1 - x q^{2n+1}) } \notag \\
  & \\
  & \times \left\{ \frac{ (xq)^e - (xq)^{d} }{ 1 - (xq)^d } q^{ne} 
		 + \frac{ 1 - (xq)^{e} }{ 1 - (xq)^d } q^{n(e-d)}\right\}
\end{align*}

The second part of our assumption is that the terms in the curly braces 
on either side of the above equation must match one for one.  
This is only possible when $e = d$ or $2e = d$.  
In the former possibility, the first fractions vanish and the seconds match, 
and in the latter the first fraction of either side match the second of the opposite side.  
In either case, we obtain
\begin{equation*}
%\label{eqOddAlphaZero}
  {}_dt\alpha_0(x) = \frac{ ( (xq^2)^d; q^{2d})_\infty }{ ( (xq)^d; q^{d})_\infty ( xq^2; q^{2})_\infty } 
  \; {}_dt\alpha_0(0) \textrm{, and}
\end{equation*}
\begin{equation*}
%\label{eqOddBetaZero}
  {}_dt\beta_0(x) = - \frac{ ( (xq^2)^d; q^{2d})_\infty }{ ( (xq)^d; q^{d})_\infty ( xq^2; q^{2})_\infty } 
  \; {}_dt\alpha_0(0) \textrm{.  }
\end{equation*}

Finally, \eqref{eqOddEvenSeriesInitEmpty} gives ${}_dt\alpha_0(0) = 1$, 
and the construction of the series is complete.  

\section{Future Research}
\label{secFutureResearch}

We indicate some problems for further research along these lines.

First of all, the results have quite a lot of missing cases.  
Theorem \ref{thmOddEvenFull} is valid only for $e = d$ or $2e = d$, 
one would like to have identities for $e = 1, \ldots, d$.  
The identities are simply wrong for other $e$'s.  
The reason for this is that the initial conditions \eqref{eqOddEvenSeriesInitZero}
are not easily met unless $e = d$ or $2e = d$.  
If one can find functions $A(x; q)$ and $B(x; q)$ such that
\begin{equation*}
%\label{eqOpenEvenInitZero}
  \sum_{n \geq 0} {}_dc\alpha[0]_n(x; q) B(xq^{2n-1}; q) - {}_dc\beta[0]_n(x; q) A(xq^{2n+1}; q) = 0
\end{equation*}
and
\begin{equation*}
%\label{eqOpenOddInitZero}
  \sum_{n \geq 0} {}_dt\alpha_n(x; q) A(xq^{2n}; q) - {}_dt\beta_n(x; q) B(xq^{2n}; q) = 0
\end{equation*}
at the same time, with $A(0; q) = B(0; q) = 1$, the construction still works.  
However, as indicated in section \ref{secConstruction}, 
we cannot make either series telescope in the sense that 
\[
  {}_dc\alpha[0]_n(x; q) B(xq^{2n-1}; q) = {}_dc\beta[0]_n(x; q) A(xq^{2n+1}; q)
\]
and 
\[
  {}_dt\alpha_n(x; q) A(xq^{2n}; q) - {}_dt\beta_n(x; q) B(xq^{2n}; q).  
\]
As one easily verifies, this leads to inconsistencies unless $e = d$ or $2e = d$.  
Thus, even if we have series for the missing cases, 
Jacobi's triple product identity may not be readily applicable, 
and the identities may not be as nice as Theorem \ref{thmOddEvenFull}.  
In order for the initial conditions to work, one may need bibasic series machinery, 
because ${}_dc\alpha[f]_n(x, q)$'s etc. are bibasic terms.  

Another issue with the construction in section \ref{secConstruction} is that it takes too long by hand.  
The nature of computations indicate that most of the process can be automated.  
A computer program which takes descriptions of various partitions as inputs, 
and producing series as generating functions will be highly valuable.  
Then, a whole bunch of Rogers-Ramanujan type theorems may be obtained effortlessly.  

The construction with minimal twists as necessary proves many well-known results in
literature such as Rogers-Ramanujan-Gordon theorem for overpartitions~\cite{Lvj-Gordon, Chen-RRGovptn}, 
Bressoud's generalization of Rogers-Ramanujan-Gordon identities for all moduli~\cite{Bres-RRGallmod}, 
Bressoud's theorem for overpartitions~\cite{Chen-BrOvptn}.  
These shall be demonstrated in separate notes.  

With a little more aid of linear algebra, 
the construction works well with partititons where $f_i + f_{i+1} + \cdots + f_{i+d-1} < k$,
plus some conditions on $f_1, \ldots, f_{d-1}$
up to the point of verifying the initial conditions.  
The obtained series contain middle $q-$multinomial coefficients when $d > 2$
(for instance $\begin{bmatrix} 2n+1 \\ n+1, n \end{bmatrix}$ comes into stage when $d = 3$, 
$\begin{bmatrix} 3n \\ n, n, n \end{bmatrix}$ when $d = 4$ etc), 
rendering the series product identities unfriendly.  
Notice that for $d = 3$ this resembles Schur's partition theorem~\cite{Schur-PtnThm} 
without the condition on multiples of three.  

Another open problem is to adapt the construction so that 
it works with partitions with multiplicity conditions
for parts that are two or more apart, 
such as G\"{o}llnitz-Gordon identities~\cite{Gollnitz, Gordon-GG}, 
or Schur's partition theorem~\cite{Schur-PtnThm}.  
The main challenge here is to guess the form of the terms in the series.  

Finally, unless $d = 1$ or $d = 2$, 
the Gordon-marking~\cite{Kur} does not help to find multiple series 
as alternative generating functions as in the case of 
Andrews-Gordon identities~\cite{Andrews-PNAS} 
or generalizations such as~\cite{AndrewsParity, Bres-RRGallmod}.  
Can one find multiple series with all positive coefficients as generating functions for $d > 2$?  

% Now define
% \[
%   A(x; q) := {}_dt\alpha_0(x; q) \frac{ ((xq)^d;  q^d)_\infty (xq^2;q^2)_\infty }{ ((xq^2)^d; q^{2d})_\infty }
%   \textrm{, and }
%   B(x; q) := - {}_dt\beta_0(x; q) \frac{ ((xq)^d;  q^d)_\infty (xq^2;q^2)_\infty }{ ((xq^2)^d; q^{2d})_\infty }
% \]

% \begin{equation}
% \label{eqEvenSeries}
%   {}_dC_{dk+e, da+f}(x; q) 
%   := \displaystyle \sum_{n \geq 0} {}_dc\alpha[f]_n(x; q) \; \left( q^{-n} \right)^{da+f} 
%     + {}_dc\beta[f]_n(x; q) \; \left( xq^{n+1} \right)^{da+f}, 
% \end{equation}
% and 
% \begin{equation}
% \label{eqOddSeries}
%   {}_dT_{dk+e, da}(x; q) 
%   := \displaystyle \sum_{n \geq 0} {}_dt\alpha_n(x; q) \; \left( q^{-n} \right)^{da} 
%     + {}_dt\beta_n(x; q) \; \left( xq^{n+1} \right)^{da}, 
% \end{equation}

\bibliographystyle{amsplain}

\end{document}